\documentclass[12pt,reqno]{amsart}
\usepackage[utf8]{inputenc}
\usepackage{amsmath}
\usepackage{amssymb}
\usepackage{amsthm}
\usepackage{xcolor}
\usepackage{fourier}
\usepackage{fullpage}
\usepackage{hyperref}
\usepackage{tikz}
\usepackage{soul}
\usepackage{mathrsfs}
\usepackage{enumitem}
\usepackage{bm}
\usepackage{mathtools}
\usepackage{tikz}
\usepackage{amssymb,amsmath,amstext,amsgen,amsfonts,amsthm,verbatim,hyperref,fullpage,enumitem}

\newcommand{\cH}{{\mathcal H}}

\newcommand{\He}{{\mathbb{H}}}

\newcommand{\ve}{{\varepsilon}}
\newcommand{\R}{{\mathbb{R}}}
\newcommand{\C}{{\mathbb{C}}}

\newcommand{\N}{\mathbb{N}}

\newcommand{\zero}{\mathbf{0}}
\DeclareMathOperator{\diam}{diam}

\newcommand{\stm}{\setminus}
\newcommand{\supp}{\operatorname{supp}}

\numberwithin{equation}{section}
\newtheorem{thm}{Theorem}[section]

\newtheorem{lem}[thm]{Lemma}

\newtheorem{cor}[thm]{Corollary}
\theoremstyle{definition}

\newtheorem{prop}[thm]{Proposition}
\theoremstyle{definition}

\newtheorem{defn}[thm]{Definition}
\theoremstyle{definition}

\theoremstyle{definition}

\newtheorem{remark}[thm]{Remark}

\numberwithin{equation}{section}

\title{On singular integrals with non-negative kernels in the Heisenberg group}

\author{Vasileios Chousionis}
\address{Department of Mathematics, University of Connecticut}
\email{vasileios.chousionis@uconn.edu}

\author{Sean Li}
\address{Department of Mathematics, University of Connecticut}
\email{sean.li@uconn.edu}

\author{Lingxiao Zhang}
\address{Department of Mathematics, University of Connecticut}
\email{lingxiao.zhang@uconn.edu}

\thanks{V.~C.\ was supported by  NSF grant 2247117. L.~Z. is supported by AMS-Simons travel grant. }

\begin{document}
\begin{abstract} In this paper we revisit nonnegative kernels in the first Heisenberg group $\He$, and in particular we further study the family
$$K_\alpha(x,y,z)= \frac{|z|^{\alpha/2}}{\|(x,y,z)\|_{H}^{\alpha+1}}, \quad \alpha>0,$$
which was introduced in \cite{CL}.

We first show that if $E \subset \He$ is a $1$-Ahlfors regular set and the SIO associated with the kernel $K_4$ 
is $L^2(E)$-bounded, then $E$ is  contained in a $1$-Ahlfors regular curve. Combined with the converse implication which was obtained by F\"assler and Orponen in \cite{FO1dim}, our result provides a characterization of uniform $1$-rectifiability in the Heisenberg group via the $L^2$-boundedness of a singular integral.

We also give a negative answer to a question of F\"assler and Orponen from \cite{FO1dim} by showing that for any $\alpha \in (0,2)$ there exists a $1$-Ahlfors regular curve $E_a$ such that the operators associated with the kernels $K_\alpha$ are not bounded in $L^2(E_\alpha)$.

We finally show that there exists a $1$-Ahlfors regular and purely $1$-unrectifiable set $E$ such that the singular integral associated with $|x| \|(x,y,z)\|^{-2}$ is $L^2(E)$ -bounded.

\end{abstract}
\maketitle
\section{Introduction and preliminaries}
Some of the most celebrated results at the interface of geometric measure theory and harmonic analysis concern the relations of rectifiability with the $L^2$ boundedness of singular integral operators (SIOs). For example, if $E \subset \C$ is $1$-Ahlfors regular, the Cauchy transform
$$C_E f\,(z)=\int_E \frac{f(w)}{z-w} d \mathcal{H}^1 (w),$$
is bounded in $L^2(E):=L^2(\mathcal{H}^1|_E)$ if and only if $E$ is contained in a $1$-Ahlfors regular curve. We say a 1-Ahlfors regular set is {\em uniformly 1-rectifiable} if it is contained in a 1-Ahlfors regular Lipschitz curve (for equivalent definitions of uniform rectifiability that extend to higher dimensions, see \cite{DS1}). 

The sufficient condition is due to David \cite{D2} and it even holds for more general smooth antisymmetric kernels. The necessary condition is a landmark result due to Mattila, Melnikov and Verdera \cite{MMV}. It is a remarkable fact that the proof in \cite{MMV} depends crucially on a special subtle positivity property of the Cauchy kernel related to an old notion of curvature named after Menger, see e.g., \cite{MeV,MMV} as well as the book \cite{tolsabook}. 

Very little is known about the behavior of SIOs associated with other $-1$-homogeneous, 1-dimensional Calder\'on--Zygmund kernels on 1-Ahlfors regular sets in $\C$. We call a kernel \emph{good} if its associated SIO is bounded on $L^2(E)$ exactly when $E$ lies in a 1-regular curve. Interestingly, all known examples---both good and bad---are related to 
\[
k_n(z) = \frac{x^{2n-1}}{|z|^{2n}}, \quad z=x+iy\in \C\setminus\{0\}, \; n\in\N.
\]
It was shown in~\cite{CMPT} that all $k_n$ with $n>1$ are also good, giving the first nontrivial examples beyond the Cauchy kernel.

Consider now the family $\kappa_t(z) = k_2(z) + t\,k_1(z)$. By \cite{CMPT,MMV}, $\kappa_t$ is good for $t>0$. Chunaev~\cite{chun} extended this to $t\leq -2$, and Chunaev, Mateu, and Tolsa~\cite{chumt} proved that the kernels $\kappa_t$ are good for $t\in(-2,-\sqrt{2})$. In contrast, Huovinen~\cite{huo} and Jaye--Nazarov~\cite{JN} constructed subtle counterexamples showing that for $t=-1$ and $t=-3/4$, boundedness of the associated SIO does \emph{not} imply rectifiability. 

In this paper, we will be interested in SIOs and $1$-rectifiability in the Heisenberg group. The \emph{Heisenberg group} $\He$ is  $\mathbb{R}^3$ equipped with the group operation
\begin{equation}\label{eq:Heis}
p \cdot q = (x+x', y+y', z+z' + \tfrac{1}{2}(xy' - yx')),
\end{equation}
for $p=(x,y,z)$ and $q=(x',y',z') \in \mathbb{R}^3$.  We will consider the Heisenberg group as a metric space $(\He,d)$ where
\begin{displaymath}
d:\He\times \He\to [0,\infty),\quad
d(p,q):= \|q^{-1} \cdot p\|,
\end{displaymath}
and $\|\cdot\|$ denotes the Koranyi norm in $\He$:
\begin{align*}
\|(x,y,z)\|= ((x^2 + y^2)^2 + z^2)^{1/4}.
\end{align*}
We will also frequently use the notation,
\begin{align*}
  NH(x,y,z) = |z|^{1/2},
\end{align*}
where $NH$ stands for non-horizontal.  
The metric $d$ is  homogeneous with respect to the \textit{dilations}
\begin{displaymath}
\delta_r:\He \to \He,\quad \delta_r((x,y,z))
=(rx,ry,r^2 z),\quad \mbox{ for }r>0.
\end{displaymath}

\begin{defn}
\label{czkernel}
A continuous function $K: \mathbb{H}\backslash \{0\} \to \mathbb{R}^d$ is called a \textit{$1$-dimensional Calder\'on-Zygmund (CZ) kernel} in $\mathbb{H}$ if there exist $\kappa \in (0,1),\beta \in (0,1], C_K \geq 1,$ such that:
\begin{itemize}
    \item (Growth condition) 
    \begin{equation}
    \label{czgrowth}
|K(p)| \leq C_K \frac{1}{\|p\|}, \quad \forall p\in \mathbb{H}\backslash \{0\},
    \end{equation}
    \item (H\"older continuity)
    \begin{equation}
    \label{czhol}
    |K(p_1)-K(p_2)| \leq C_K \frac{\|p_2^{-1} \cdot p_1\|^\beta}{\|p_1\|^{1+\beta}}, \quad \forall p_1, p_2\in \mathbb{H}\backslash \{0\} \text{ with } d(p_1, p_2) \leq \kappa \|p_1\|.
    \end{equation}
\end{itemize}
\end{defn}

A Borel set $E \subset \He$ is called $1$-Ahlfors regular if there exists $C>0$ such that
\begin{equation}
\label{eq:ADR}
C^{-1}r \leq \mathcal{H}^1(B(p,r) \cap E) \leq Cr, \quad \forall p\in E, 0<r \leq \diam(E),
\end{equation}
where $\mathcal{H}^1$ denotes the $1$-dimensional Hausdorff measure in $\He$ (with respect to the metric $d$).

If $K$ is a $1$-dimensional CZ kernel in $\He$ and $E\subset \He$ is a $1$-Ahlfors regular set then H\"older's inequality and \eqref{eq:ADR} imply that the \textit{truncated singular integrals}
$$
T_{\ve}f(p) = \int_{E \cap B(p,\ve)^c} K(p^{-1}\cdot q)f(q)\,d\mathcal{H}^1(q
), \quad p\in \mathbb{H},
$$
are well defined for $f\in L^p(E):=L^p(\mathcal{H}^1|_{E})$ and $\ve>0$ when $p \in (1,\infty)$. 

\begin{defn}\label{SIO}
Given a $1$-dimensional CZ kernel $K$ and a $1$-Ahlfors regular set in $\He$, we say that the singular integral operator (SIO) $T$ associated to $K$ and $E$ is bounded on $L^p(E), p \in (1,\infty),$ if the operators
$$
f\mapsto T_{\ve}f
$$
are bounded on $L^p(E)$ with constants independent of $\ve>0$. We set
$$\|T\|_p=\inf\left\{C>0:  \|T_\ve f\|_{L^p(E)} \leq C \|f\|_{L^p(E)}: f \in L^p(E), \ve>0\right\}.$$We also define the \textit{principal value of $f$ at $p$} to be
\begin{align*}
    \mathrm{p.v.}Tf(p) = \lim_{\ve \to 0} T_\ve f(p),
\end{align*}
when the limit exists. 
\end{defn}

Recall that all $1$-dimensional kernels in $\C$ that we discussed so far are odd, which is quite natural. Indeed, consider the $1$-dimensional Calder\'on--Zygmund kernel $$k(x,y)=|x-y|^{-1}$$ which lacks local integrability along the diagonal. In this case, one has $\int_I k(x,y)\,dy=\infty$ for every open interval $I\subset\R$. Thus, in the Euclidean setting, defining SIOs on lines or other regular $1$-dimensional sets hinges on the cancellation properties of the kernel. 

It was discovered by the two first named authors in \cite{CL} that the situation in the Heisenberg group $\He$ is entirely different. In \cite{CL} the following family of $-1$-homogeneous kernels in $\He$ was introduced: 
\begin{equation}
\label{seankernels}
K_\alpha (p)=\frac{NH(x,y,z)^\alpha}{\|(x,y,z)\|^{\alpha+1}}=\frac{|z|^{\alpha/2}}{\|p\|_{H}^{\alpha+1}}, \quad \alpha>0.
\end{equation}
The kernels $K_\alpha$ can be thought of as weighted versions of the Riesz kernel corresponding to the vertical component of $\He$. In \cite{CL} 
it was proved that if $E$ is contained in a $1$-regular curve, then the SIO associated with $K_8$ is bounded in $L^2(E)$. Conversely,  the $L^2(E)$-boundedness of the SIO associated with $K_2$ implies that $E$ is contained in a $1$-regular curve. These were the first non-Euclidean examples of kernels with such properties.  We stress that unlike the Euclidean case, where all known kernels related to rectifiability are odd, the kernels $K_\alpha$ are {\em even} and {\em nonnegative.} 

To build some intuition for why the positive kernels in \eqref{seankernels} might be bounded on Lipschitz curves, recall that Rademacher’s theorem ensures such curves in $\R^n$ are infinitesimally close to affine lines, on which antisymmetric kernels cancel. This cancellation is what controls the singularity in the Euclidean case. In the Heisenberg group, an analogue due to Pansu~\cite{pansu} states that Lipschitz curves are infinitesimally close to \emph{horizontal lines}, which are defined as left cosets of $1$-dimensional subgroups of $\He$ contained in $\{(x,y,0) \in \He: x,y \in \R \}.$ Since  $K_\alpha(p^{-1}q)$ vanish  precisely when $p,q$ lie on a horizontal line, the singularity is again under control in our case.

We will now focus on $-1$-homogeneous $1$-dimensional CZ kernels in $\He$ which all exhibit very different behavior. Recall that a function $f: \He  \stm \{\zero\} \to \R^d$ is called \textit{$\lambda$-homogeneous}, for $\lambda \in \R$, if
$$f(\delta_t p)=t^{\lambda} f(p).$$ Clearly the kernels $K_\alpha$, as in \eqref{seankernels} are $-1$-homogeneous. 

Our first theorem concerns the kernel $K_4$ from \ref{seankernels} and reads as follows.
\begin{thm}
\label{mainthm1} Let $E$ be a $1$-Ahlfors regular set. If the singular integral associated with $K_4$ and $E$ is bounded in $L^2(E)$; i.e., if the truncated singular integrals
$$T^4_\ve f\,(p):=\int_{E \stm B(p,\ve)} K_4(p^{-1} \cdot q) f(q) d \mathcal{H}^1(q)$$ are uniformly bounded in $L^2(E)$, then $E$ is contained in a $1$-Ahlfors regular curve.
\end{thm}

The proof of Theorem \ref{mainthm1}, which can be found in Section \ref{sec:mainthm1}, is based on a modification of the argument used in \cite[Theorem 1.3]{CL}. The key difference is that, rather than relying directly on the $L^2(E)$-boundedness of $T^4$, we exploit the fact that $T^4$ is locally $L^1(E)$-bounded on indicator functions, a property that follows from its $L^2(E)$-boundedness. Thus, one may ask if $L^1$-boundedness might be the proper condition to consider, which may make sense given the positive nature of the kernel. In Section \ref{sec:l1unbound} we prove that this is not the case, by showing that there exists a $1$-Ahlfors regular curve for which the SIO associated with $K_4$ is not $L^1$-bounded.

F\"assler and Orponen~\cite{FO1dim} recently performed a comprehensive study of $1$-dimensional kernels in the Heisenberg group and among many other things they showed that $K_4$ defines an $L^2(E)$ bounded operator whenever $E$ is contained in a $1$-Ahlfors regular curve, \cite[Theorem 1.7]{FO1dim}.
Theorem \ref{mainthm1} combined with the aforementioned result from \cite{FO1dim} provides a characterization of uniform $1$-rectifiability in the Heisenberg group. This is the first example of a good kernel beyond Euclidean spaces.
\begin{cor}
\label{mainchar}
Let $E$ be a $1$-Ahlfors regular set. The set $E$ is contained in  a $1$-Ahlfors regular curve if and only if the singular integral associated to $K_4$ and $E$ is bounded in $L^2(E)$.
\end{cor}

F\"assler and Orponen in \cite{FO1dim} proved that any smooth $1$-dimensional kernel in $\He$ which is odd or horizontally odd (i.e. it satisfies $K(-x,-y,z)=-K(x,y,z)$ defines $L^2$-bounded SIOs on any $1$-Ahlfors regular curve. They also asked, \cite[Question 1]{FO1dim}, if for $-1$-homogeneous kernels the cancellation conditions mentioned above (oddness or horizontal oddness) could be relaxed by only assuming that the operator is $L^2$-bounded on horizontal lines with uniform constants. This is a natural question which was partly motivated by the fact that the kernels $K_{\alpha}, \alpha \geq 4,$ define $L^2$-bounded SIOs  on regular curves, while they are neither odd nor horizontally odd. However, as mentioned earlier, the SIOs associated with the kernels $K_\alpha$ are, trivially, $L^2$-bounded on horizontal lines with uniform constants. 

We use again members of  the family \eqref{seankernels} to give a negative answer to this question.
\begin{thm}
\label{FOquestion}
Let $\alpha \in (0,2)$. There exists a $1$-Ahlfors regular curve $E_\alpha$ such that the singular integral associated with $K_\alpha$ and $E_\alpha$ is not bounded in $L^2(E_\alpha)$.
\end{thm}
The curves $E_a$ are constructed as horizontal lifts of von Koch-like curves in the plane, similarly to \cite{seanmarkov}. While it may be possible to extend the result to $\alpha \in [2,4)$, restricting to $\alpha < 2$ allows us to take $E_\alpha$ to be lifts of Lipschitz graphs, which drastically simplifies the proof. The proof of Theorem \ref{FOquestion} can be found in Section \ref{sec:FOquestion}.

We finally consider the kernel \begin{equation}
\label{eq:kerunrec}
K_b(x,y,z)=\frac{|x|}{\|(x,y,z)\|^2}.
\end{equation}Note that the kernel $K_b$ is $-1$-homogeneous and it follows as in \cite[Proposition 2.5]{CLZh} that it is a $1$-CZ kernel.
We will show that the SIO associated with $K_b$ behaves very differently from the SIO associated with $K_4$. For the next theorem, recall that a set $E \subset \He$ is called purely $1$-unrectifiable if $\mathcal{H}^1(E \cap f(\R))=0$ for any Lipschitz function $f:\R \to \He$.
\begin{thm}\label{mainthm2}
There exists a $1$-Ahlfors regular and  purely $1$-unrectifiable set  $E$ such that the singular integral associated with $K_b$ and $E$ is bounded in $L^2(E)$, i.e., the truncated singular integrals
$$T^b_\ve f\,(p):=\int_{E \stm B(p,\ve)} K_b(p^{-1} \cdot q) f(q) d \mathcal{H}^1(q)$$ are uniformly bounded in $L^2(E)$. Moreover, for every $p \in E$, we have $(p^{-1}E\backslash \{0\}) \subset \{K_b\neq 0\}$.
\end{thm}
The proof of Theorem \ref{mainthm2} can be found in the last section of our paper; Section \ref{sec:mainthm2}. We note that, due to the characteristics of $\He$, the set $E$ is obtained in a rather simple way. We also note that it is a much harder problem to construct such sets in the Euclidean case, see \cite{JN} and \cite{huo}.



\section{A SIO whose $L^2$ boundedness characterizes uniform $1$-rectifiability.} \label{sec:mainthm1}

\begin{proof}[Proof of Theorem \ref{mainthm1}]For $\alpha \in (0,1)$ and $r > 0$, we let $\Sigma(\alpha,r)$ to be the set of triples of points $(p_1,p_2,p_3) \in E$ so that
\begin{align*}
  \alpha r \leq d(p_i,p_j) \leq r, \qquad \forall i \neq j.
\end{align*}
We also let $\Sigma(\alpha) = \bigcup_{r > 0} \Sigma(\alpha,r)$. For $x,y \in E$, we let
\begin{align*}
  \Sigma(\alpha,r;x) &:= \{(y,z) \in E^2 : (x,y,z) \in \Sigma(\alpha,r)\}, \\
  \Sigma(\alpha;x,y) &:= \{z \in E : (x,y,z) \in \Sigma(\alpha)\}.
\end{align*} The \textit{Menger curvature } of  $p_1,p_2,p_3$ in $E$, is denoted by $c(p_1,p_2,p_3) \in \R$ and is defined as
\begin{align*}
  c(p_1,p_2,p_3) = \frac{1}{R},
\end{align*}
where $R$ is the radius of the circle in $\R^2$ passing through a triangle defined by the vertices $p_1',p_2',p_3' \in \R^2$ where $d(p_i,p_j) = |p_i' - p_j'|$. 

By a recent result of F\"assler and Violo, \cite[Corollary 3.20]{FV} it suffices to prove that there is some $\alpha > 0$ so that
$$\iiint_{\Sigma(\alpha) \cap (B(p_0,R) \cap E)^3} c(p_1,p_2,p_3)^2 ~dp_1 ~dp_2 ~dp_3 \lesssim_{\alpha} R, \qquad \forall p_0 \in E, R > 0.$$
We would like to point out that in \cite{CL} the two first named authors had mistakenly attributed the previous result to Hahlomaa, from \cite{hah}. Although, it might have been known to some people, see for example \cite{schul}, the proof was nowhere published before the work of F\"assler and Violo.

By \cite[Proposition 4.5]{CL} it is enough to show that there is some $\alpha > 0$ so that
\begin{equation}
\label{eqn:cl-prop4.5}
\iiint_{\Sigma(\alpha)\cap (B(p_0,R)\cap E)^3} \frac{\gamma_1(p_1, p_2, p_3)^2 \gamma_2(p_1, p_2, p_3)^2}{\diam (p_1, p_2, p_3)^2} \,dp_1\,dp_2\,dp_3\lesssim_{\alpha} R, \qquad \forall p_0 \in E, R > 0.
\end{equation}
We will actually show that \eqref{eqn:cl-prop4.5} holds  for any $\alpha \in (0,1)$. For simplicity of notation we let $T^4:=T$ and $\|T^4\|_2=\|T\|$. Denote
$$
\chi_R = \chi_{B(p_0,R)\cap E}.
$$
We have
$$
\langle T_\ve \chi_R, \chi_R\rangle_{L^2(E)} \leq \|T_\ve\|_{2\to 2} \|\chi_R\|_{L^2(E)}^2 \lesssim_{\|T\|} R.
$$
Thus
\begin{multline}
\int_R T_\epsilon \chi_R(p) ~dp = \langle T_\ve \chi_R, \chi_R\rangle_{L^2(E)} = \int_{E\cap B(p_0,R)} \int_{E\cap B(p_0, R)\cap B(p_1, \ve)^c} \frac{NH(p_1^{-1} p_2)^4}{d(p_1, p_2)^5}\,dp_2\,dp_1 \\ \lesssim_{\|T\|} R.
\label{eq:L1?}
\end{multline}
Let $\ve \to 0$, by Fatou's lemma,
$$
\iint_{(E\cap B(p_0,R))^2} \frac{NH(p_1^{-1} p_2)^4}{d(p_1, p_2)^5}\,dp_2\,dp_1 \lesssim_{\|T\|} R.
$$
Let $\alpha \in (0,1)$ and note that:
\begin{align*}
\left|\Sigma(\alpha;x p_1, p_2) \cap E\right| &\leq \left|\left\{p_3\in E: d(p_3, p_1)\leq \frac{1}{\alpha} d(p_1, p_2)\right\}\right| =\left|E\cap B\left(p_1, \frac{1}{\alpha} d(p_1, p_2)\right)\right|\lesssim_{ \alpha} d(p_1, p_2).
\end{align*}
Therefore,
\begin{align*}
\iiint_{\Sigma(\alpha)\cap (B(p_0,R)\cap E)^3} &\frac{NH(p_1^{-1} p_2)^4}{d(p_1, p_2)^6}\,dp_1 \,dp_2\,dp_3 \\
&\leq \iint_{(E\cap B(p_0,R))^2} \frac{NH(p_1^{-1} p_2)^4}{d(p_1, p_2)^5} \int_{E\cap \Sigma(\alpha, p_1, p_2)} \frac{dp_3}{d(p_1, p_2)}\,dp_1\,dp_2\\
&\lesssim_{ \alpha} \iint_{(E\cap B(p_0,R))^2} \frac{NH(p_1^{-1} p_2)^4}{d(p_1, p_2)^5}\,dp_1\,dp_2\\
&\lesssim_{\|T\|} R.
\end{align*}
Hence,
\begin{align*}
\iiint_{\Sigma(\alpha)\cap (B(p_0,R)\cap E)^3} &\frac{\gamma_1(p_1, p_2, p_3)^2 \gamma_2(p_1, p_2, p_3)^2}{\diam (p_1, p_2, p_3)^2} \,dp_1\,dp_2\,dp_3\\
&\lesssim \sum_{\sigma \in S_3} \iiint_{\Sigma(\alpha)\cap (B(p_0,R)\cap E)^3} \frac{NH(p_{\sigma(1)}^{-1}p_{\sigma(2)})^4}{d(p_{\sigma(1)},p_{\sigma(2)})^6} \,dp_1\,dp_2\,dp_3\\
&\lesssim_{ \|T\|, \alpha} R.
\end{align*}
\end{proof}

\section{$L^1$ unboundedness}
\label{sec:l1unbound}
The proof of Theorem \ref{mainthm1} essentially uses a local $L^1$-boundedness property of the operator rather than $L^2$-boundedness. Indeed, this is most apparent in \eqref{eq:L1?}. So, one could ask if $T^4$ is $L^1$-bounded on $1$-Ahlfors regular curves. In this section we prove that this is not the case, as there is a 1-Ahlfors regular curve for which the SIO associated with $K_4$ is not $L^1$-bounded.

Let $\gamma$ be the horizontal lift of the curve $t \mapsto (t,t \sin \log t), t>0$. First note that
$$\left| \frac{d}{dt} (t \sin \log t) \right| = \left| \sin \log t + \cos \log t \right| \leq 2,$$
so $\gamma$ is the lift of a Lipschitz graph and is indeed Lipschitz and 1-regular. Also note that $\|\gamma(s)^{-1}\gamma(t)\| \asymp |t-s|$.

\begin{prop}
  For any interval $I$,
  \begin{align*}
    \int_\R \int_I K_4(\gamma(s)^{-1}\gamma(t)) ~ds ~dt = \infty.
  \end{align*}
\end{prop}

\begin{proof}
  By Tonelli, it suffices to prove for any $s$ that
  \begin{align*}
    \int_\R K_4(\gamma(s)^{-1}\gamma(t)) ~dt = \infty.
  \end{align*}
  We may assume without loss of generality that $s > 0$.

  By the area form, we have
  \begin{multline*}
    NH(\gamma(s)^{-1}\gamma(t))^2 = \frac{1}{2}\int_s^t ~x ~dy - y ~dx = \frac{1}{2} \int_s^t x \left( \cos \log x - \sin \log x \right) ~dx \\
    = \frac{1}{\sqrt{8}} \int_s^t x \cos \left( \log x - \frac{\pi}{4} \right) ~dx = \frac{1}{\sqrt{8}} \int_{\log s}^{\log t} e^{2u} \cos \left( u - \frac{\pi}{4} \right) ~du.
  \end{multline*}
  For any $n \in \N$, we have that $\cos (u - \tfrac{\pi}{4}) \geq \frac{1}{\sqrt{2}}$ for all $u \in [2\pi n, 2 \pi n + \tfrac{\pi}{2}]$. Thus,
  \begin{align*}
    \int^{2\pi n + \frac{\pi}{2}}_{2\pi n-\frac{\pi}{4}} e^{2u} \cos \left( u - \frac{\pi}{4} \right) ~du \geq \int^{2\pi n + \frac{\pi}{2}}_{2\pi n} e^{2u} \cos \left( u - \frac{\pi}{4} \right) ~du  \geq e^{4 \pi n} \cdot \frac{1}{\sqrt{2}} \cdot \frac{\pi}{2}.
  \end{align*}
  On the other hand, for any $0 \leq z < 2 \pi n-\frac{\pi}{4}$, we have
  \begin{align*}
    \left| \int_z^{2\pi n - \frac{\pi}{4}} e^{2u} \cos \left( u - \frac{\pi}{4} \right) ~du \right| \leq \int_0^{2\pi n} e^{2u} ~du \leq \frac{1}{2}e^{4\pi n}.
  \end{align*}
  This gives for $c = \frac{\pi}{\sqrt{8}} - \frac{1}{2} > 0$ that
  \begin{align*}
    \int_{e^s}^{z}  e^{2u} \cos \left( u - \frac{\pi}{4} \right) ~du \geq c e^{4\pi n}, \qquad \forall z \in \left[ 2 \pi n + \tfrac{\pi}{2}, 2\pi n + \tfrac{3\pi}{4} \right].
  \end{align*}
  when $n$ is sufficiently large. Going back to the $x$ domain, this gives
  \begin{align*}
    NH(\gamma(s)^{-1}\gamma(t))^2 = \int_{s}^t  x \cos \left( \log x - \frac{\pi}{4} \right) ~dx \geq c e^{4\pi n}, \qquad \forall t \in I_n := \left[ e^{2 \pi n + \tfrac{\pi}{2}}, e^{2\pi n + \tfrac{3\pi}{4}} \right].
  \end{align*}
  As $\|\gamma(s)^{-1} \gamma(t)\| \asymp |s-t| \asymp e^{2\pi n}$ for $t \in I_n$ when $n \geq N$ for some $N$ sufficiently large, we now get
  \begin{align*}
    K_4(\gamma(s)^{-1}\gamma(t)) = \frac{NH(\gamma(s)^{-1}\gamma(t))^4}{\|\gamma(s)^{-1}\gamma(t)\|^5} \gtrsim \frac{c^2e^{8\pi n}}{e^{10 \pi n}} \gtrsim e^{-2\pi n}, \qquad \forall t \in I_n.
  \end{align*}
  Finally, as $|I_n| \gtrsim e^{2\pi n}$, we have
  \begin{align*}
    \int_s^\infty K_4(\gamma(s)^{-1}\gamma(t)) ~dt \geq \sum_{n \geq N} \int_{I_n} K_4(\gamma(s)^{-1}\gamma(t)) ~dt \gtrsim \sum_{n \geq N} e^{-2\pi n} |I_n| \gtrsim \sum_{n \geq N} 1 = \infty.
  \end{align*}
\end{proof}

\section{A SIO which is not $L^2$-bounded on a regular curve} \label{sec:FOquestion}

This section gives a negative answer to a question of F\"assler and Orponen: \cite[Question 1]{FO1dim}. We will show that there exists a $1$-dimensional CZ kernel such that the corresponding singular integral is $L^2$-bounded along horizontal lines with uniform constants, but the singular integral along a certain $1$-regular curve $E$ is not $L^2$- bounded.

The relevant kernels will again come from the family of kernels \eqref{seankernels}. More precisely, we will show for any $\alpha<2$ there exists a $1$-regular curve $E_{\alpha}$, such that the SIO associated to $K_{a}$ (as in \eqref{seankernels}) and $E$ is not bounded on $L^2(E)$.

\begin{center}
\begin{tikzpicture}[scale=1]
\draw (-10.5,0) -- (2.5,0);
\draw (-10.1, 0) -- (-9.7,0.15) --(-9.25, 0) -- (-8.8, -0.15) -- (-8.4, 0);
\draw (0.4,0) -- (0.8, 0.15) -- (1.25, 0) -- (1.7, -0.15) -- (2.1, 0);
\draw (-7.7,0.12) node[anchor=west]{\scriptsize $\theta_n$};
\draw (-7.67,0.17) -- (-7.42, 0.45) -- (-7,0.5) -- (-6.58, 0.55) -- (-6.33,0.83);
\draw (-7.6, 0.7) node[anchor=north]{\scriptsize $\theta_{n+1}$};
\draw (-5.67,0.83) -- (-5.25, 0.78) -- (-5, 0.5) -- (-4.75, 0.22) -- (-4.33, 0.17);
\draw (-3.67,-0.17) -- (-3.25, -0.22) -- (-3, -0.5) -- (-2.75, -0.78) -- (-2.33,-0.83);
\draw (-1.67,-0.83) -- (-1.42, -0.57) -- (-1,-0.5) -- (-0.58, -0.43) -- (-0.33,-0.17);
\draw (-8,0) -- (-6,1) -- (-4,0)  -- (-2,-1) -- (0,0);
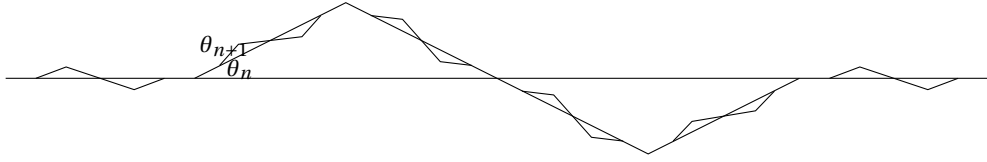
\draw (-4,-1) node[anchor=north]{\small FIGURE 1: Stage $n, (n+1)$ polygonal replacements of a stage $(n-1)$ segment in $\mathbb{C}$};
\end{tikzpicture}
\end{center}

The curves $E_{\alpha}$ will be obtained as a lift in $\mathbb{H}^1$ of a von Koch-like curve $\Gamma\subseteq \mathbb{R}^2=\mathbb{C}$ similar to the graph which was constructed in \cite{seanmarkov}. We now describe the iterative construction:
\begin{enumerate}
    \item 
    Without loss of generality, start with a horizontal segment in the plane (call it stage $0$): a segment $J_0$ on the $x$-axis, parametrized from left to right. 

    \item Let $(\theta_n)$ be a decreasing sequence of positive numbers such that
\begin{equation}\label{3}
\sum_n \theta_n<\frac{1}{2}.
\end{equation}

    \item 
    At stage $n\geq 1$, each segment of the previous polygonal curve in $\mathbb{C}$ is replaced by a small six-segment “bump” (a $6$–gonal replacement). For example, given a line segment from $(0,0)$ to $(x,y)$, we can replace it with the piecewise linear curve in $\mathbb{C}$ connecting the following points:

\begin{equation}\label{horizontal}
\begin{aligned}
(0,0), \quad \left(\frac{x}{2+4\cos \theta_n}, \frac{y}{2+4\cos \theta_n} \right), \\ \left(\frac{1+\cos \theta_n}{2+4\cos \theta_n}x-\frac{\sin \theta_n}{2+4\cos \theta_n} y, \frac{1+\cos \theta_n}{2+4\cos \theta_n}y + \frac{\sin \theta_n}{2+4\cos \theta_n}x\right), \quad \left(\frac{x}{2}, \frac{y}{2}\right), \\
\left(\frac{1+3\cos \theta_n}{2+4\cos \theta_n}x+ \frac{\sin \theta_n}{2+4\cos \theta_n}y, \frac{1+3\cos \theta_n}{2+4\cos \theta_n}y - \frac{\sin \theta_n}{2+4\cos \theta_n}x\right), \\
\left( \frac{1+4\cos \theta_n}{2+4\cos \theta_n}x, \frac{1+4\cos \theta_n}{2+4\cos \theta_n}y \right), \quad (x,y),
\end{aligned}
\end{equation}
where $\theta_n>0$ is the angle between the second segment of the polygonal curve and the direction $(x,y)$. Note that each of the six segments has the same length.


Denote each piece of stage $n$ in $\mathbb{C}$ as $J_{n,i}, i=1,\dots, 6^n,$ and denote its length in $\mathbb{C}$ by $R_n$. We have $R_n=R_0\cdot \prod_{j=1}^n \frac{1}{2+4\cos \theta_j}$. Let $J_n:=\cup_{i=1}^{6^n} J_{n,i} $. The planar limit curve $\Gamma:=\Gamma((\theta_n))$ is a continuous curve obtained as the uniform limit of the polygonal approximants $J_n$. We will need the following lemma.

\begin{lem}\label{gamma}
$\Gamma$ is a Lipschitz graph.
\end{lem}

\begin{proof}
For every finite-stage polygonal approximant each segment’s angle (with respect to the $x$-axis) is bounded in absolute value by $\sum_n \theta_n$, and the polygonal curve is the graph of a function $y=y(x)$ with
$$
|y'(x)|\leq \tan(\sum \theta_n).
$$
on the differentiable pieces. The polygonal approximants converge uniformly to the planar limit curve $\Gamma$. Thus $\Gamma$ is Lipschitz with constant $L\leq \tan(\sum \theta_n)$.
\end{proof}

\item 
Since $\Gamma$ is a Lipschitz graph we can consider its \textit{horizontal lift} in $\He$. Writing points in $\mathbb{H}^1$ as $(V,z)$ with $V=x+iy\in \mathbb{C}$, the horizontal lift $E:=E((\theta_n))$ of $\Gamma$: $$w(s)= (V(s), z(s))$$ is defined so that its projection onto $\mathbb{C}$ is $\Gamma$ and that the derivative $w'$ is horizontal, i.e.,
$$
z'(s) = \frac{1}{2} \text{Im}\,\overline{ V(s)}\cdot V'(s).
$$

\item 
We give points in $\Gamma (\subseteq \mathbb{C})$ an ordering from left to right and a word expression $w=(w_n)_{n=1}^\infty$ with $w_n\in \{1,2,\dots, 6\}$. Recall that $\Gamma$ is the uniform limit of the entire stage $n$ curve as $n\to \infty$. For a point $(w_n)_n$, the coordinate $w_1$ represents that it corresponds to a point in the $w_1$-th segment when we replace $J_0$ with the $6$-gonal curve at stage $1$; once $w_n$ locates the $J_{n,i}$ segment at stage $n$, the coordinate $w_{n+1}$ represents that it corresponds to a point in the $w_{n+1}$-th segment when we replace $J_{n,i}$ with the $6$-gonal curve at stage $(n+1)$. When there is no ambiguity, we will denote points of $E (\subseteq \mathbb{H}^1)$ by the same word expression as points in $\Gamma$, i.e. $(w_n)_n$.  

\end{enumerate}

\begin{lem}
The set $E$ is a Lipschitz graph and thus it is also  $1$-Ahlfors regular.
\end{lem}

\begin{proof}
The lift $E$ of $\Gamma$ satisfies
$$
z(t) = z(0) + \frac{1}{2}\int_0^t \text{Im}\, \overline{V(s)} V'(s)\,ds.
$$
Thus 
$$
|V(t_2)-V(t_1)|\lesssim |t_2-t_1|.
$$
Combining this with Lemma \ref{gamma}, we get that $E$ is a Lipschitz graph.
\end{proof}

\begin{remark}
\label{rem:noselfintersecting} Since $E$ is a Lipschitz graph it is a non self-intersecting curve.
\end{remark}
We can now restate the main theorem of this section.
\begin{thm}\label{5.1}
Let $\alpha \in (0,2)$ and let $K_{\alpha}$ be a kernel as in \eqref{seankernels}. Denote by $T^{\alpha}$ its corresponding singular integral. There exists some $c:=c(\alpha)>0$, such that by taking $\theta_n = \frac{c}{n^{1/\alpha}}$, $$\langle T^{\alpha} \chi_E, \chi_E \rangle =\infty,$$
where $E:=E((\theta_n))$ is the $1$-Ahlfors curve that we previously defined. In particular, the operator $T^{\alpha}$ is not $L^2(E)$ bounded.
\end{thm}



The key step in the proof of Theorem \ref{5.1} is the following technical lemma.

\begin{lem}\label{5.4}
Fix arbitrary $n$. Consider arbitrary $p =(p_i), q=(q_i)\in E$ whose first $(n-1)$ coordinates agree, and $p_n=1, q_n =4$. We have
$$
|NH(p^{-1}q)|^2 \gtrsim R_n^2 \theta_n,
$$
uniformly in $n$ and $i$.
\end{lem}

\begin{center}
\begin{tikzpicture}[scale=1]
\draw (-4,0) node[anchor=south] {\scriptsize $O$};
\draw[densely dotted, fill=gray!30] (-10.5, 0) node[anchor=east] {\scriptsize $B$} -- (-9.7, 0.15) node[anchor=south]{\scriptsize $A$} -- (-8,0) -- (-8.8, -0.15) -- (-10.5, 0);
\filldraw  (-9.1, 0.05) circle (0.5pt);
\draw (-9.1, 0.05) node[anchor=south] {\scriptsize $p$};
\draw [densely dotted, thick](-7.67,0.17) -- (-4,0);
\draw (-5.83, 0.06) node[anchor=south] {\scriptsize $\ell$};
\draw (-10.5,0) -- (2.5,0);
\draw (-10.1, 0) -- (-9.7,0.15) --(-9.25, 0) -- (-8.8, -0.15) -- (-8.4, 0);
\draw (0.4,0) -- (0.8, 0.15) -- (1.25, 0) -- (1.7, -0.15) -- (2.1, 0);
\draw (-7.67,0.17) -- (-7.42, 0.45) -- (-7,0.5) -- (-6.58, 0.55) -- (-6.33,0.83);
\draw (-5.67,0.83) -- (-5.25, 0.78) -- (-5, 0.5) -- (-4.75, 0.22) -- (-4.33, 0.17);
\draw[densely dotted, fill=gray!30] (-4, 0) -- (-3.25, -0.22)  -- (-2,-1) node[anchor=north] {\scriptsize $D$} -- (-2.75,-0.78) -- (-4,0);
\draw (-3.25,-0.3) node[anchor=south] {\scriptsize $C$};
\filldraw (-2.9,-0.48) circle (0.5pt);
\draw (-2.9,-0.48) node[anchor=south] {\scriptsize $q$};
\draw (-3.67,-0.17) -- (-3.25, -0.22) -- (-3, -0.5) -- (-2.75, -0.78) -- (-2.33,-0.83);
\draw (-1.67,-0.83) -- (-1.42, -0.57) -- (-1,-0.5) -- (-0.58, -0.43) -- (-0.33,-0.17);
\draw (-8,0) -- (-6,1) -- (-4,0)  -- (-2,-1) -- (0,0);
\draw (-4,-1) node[anchor=north]{\small FIGURE 2};
\end{tikzpicture}
\end{center}

\begin{proof}
In Figure 2, the horizontal line represents a line segment in the $(n-1)$-th stage, and the figure also include the $n$-th and $(n+1)$-th stage polygonal replacements. Arguing as in  \cite[Lemma 2.7]{seanmarkov} (note that our curve is a subset of one appearing in \cite{seanmarkov}), $p,q$ belong to the two convex hulls, respectively, the shaded areas in Figure 2. The absolute value of the 3rd coordinate of$ p^{-1}q$, $NH(p^{-1}q)^2$, is the absolute value of the signed area of the region in the $\mathbb{C}$ plane formed by the part of the graph $\Gamma$ from $p$ to $q$ and travel back through a line segment from $q$ to $p$. 

To compute the area, note that every generation $i$, compared to the previous generation, generates positive and negative triangle each with area $S_i=R_i^2\sin \theta_i \cos \theta_i$. Let $J_{n-1,j}$ denote the $(n-1)$-stage segment that the first $(n-1)$ coordinates of $p$ and $q$ correspond to, i.e., the horizontal line in Figure 2. Without loss of generality, denote by $J_{n,1}, \ldots, J_{n,6}$ the six line segments that replaces $J_{n-1,j}$ in stage $n$. Without loss of generality, denote by $J_{n+1, 1}, \ldots, J_{n+1,6}$ the six segments that replaces $J_{n,1}$ in stage $(n+1)$; denote by $J_{n+1, 7}, \ldots, J_{n+1, 12}$ the six segments that replaces $J_{n,4}$ in stage $(n+1)$. By a similar argument to \cite[Lemma 2.7]{seanmarkov} (our curve is a subset of the vertically symmetrical graph from \cite{seanmarkov}), we see that
$$
\{(w_i)\in \Gamma: w_i = p_i \text{ for }i\leq n\} \subseteq \text{ convex hull of } \left(J_{n+1,1} \cup \cdots \cup J_{n+1,6}\right),
$$
and
$$
\{(w_i)\in \Gamma: w_i = q_i \text{ for }i\leq n\} \subseteq \text{ convex hull of } \left(J_{n+1,7} \cup \cdots \cup J_{n+1,12}\right).
$$

Therefore by referring to Figure 1, the $pq$ line intersects with the line $J_{n,2}$ no higher than the lowest $\frac{1}{2+4\cos \theta_{n+1}}$ portion of the segment $J_{n,2}$, and intersects with the line $J_{n,3} \cup J_{n,4}$ in the $J_{n,4}$ portion (Thus $pq$ line is below the line $\ell$ in Figure 2). Thus the positive triangle generated at stage $n$ is with area at least 
\begin{equation}
\left(1-\frac{1}{2+4\cos \theta_{n+1}}\right)S_n=\frac{1+4\cos \theta_{n+1}}{2+4\cos \theta_{n+1}}S_n.
\end{equation}

Without loss of generality, we draw only the figures of the worst-case scenario below, where $p$ on top of the horizontal line in Figure 2 (the $(n-1)$-th stage line segment) and $q$ below it.

We have
$$
|NH(p^{-1}q)^2| \geq I + I\!I,
$$
where $I$ denotes the contribution of stage $n$ to the signed area, and $I\!I$ denotes the contribution of all $>n$ stages. All the later generations $(n+j)$ are compared with the previous generation $(n+j-1)$ to generate positive or negative triangles.
As $E$ does not self-intersect, we have  
$$
I\!I \geq -\sum_{j=1}^\infty 2\cdot 6^{j-1} S_{n+j}.
$$

For $I$, it is bounded below by the positive area minus the two negative areas in worst scenario shown in Figure 3, which in turn is bounded below by the positive area minus the two negative areas in Figure 4, denoted by $I\!I\!I$.

\begin{center}
\begin{tikzpicture}[xscale=1.5, yscale=3]
\draw (-4,0) node[anchor=south] {\scriptsize $O$};
\filldraw  (-9.7, 0.1) circle (0.3pt);

\draw (-10.5,0) -- (-8,0);
\draw[fill=gray!60] (-9.7,0.1) -- (-9.73,0) -- (-8,0) -- (-7.92, 0.04);
\draw[fill=gray!60] (-3.7,-0.11) -- (-3.7, -0.15) node[anchor=west]{\scriptsize negative area \# 2} -- (-3.8, -0.1);
\draw[fill=gray!15] (-7.92, 0.04) -- (-6,1) -- (-3.8, -0.1);
\filldraw (-3.7,-0.11) circle (0.3pt);
\filldraw (-4,0) circle (0.3pt);
\draw (-8,0) -- (-6,1)  -- (-4,0)  -- (-2,-1);
\draw (-6,-1) node[anchor=north]{\small FIGURE 3};
\draw (-6,0.55) node[anchor=north]{\scriptsize positive area};
\draw (-9.7, 0.1) node[anchor=south] {\scriptsize $p$} -- (-3.7,-0.11) node[anchor=south] {\scriptsize $q$};
\draw (-9, 0) node[anchor=north]{\scriptsize negative area \# 1};

\end{tikzpicture}
\end{center}

\begin{center}
\begin{tikzpicture}[xscale=1.5, yscale=3]
\draw[fill=gray!60] (-4, 0) -- (-3.25, -0.22)  -- (-2,-1) node[anchor=north] {\scriptsize $D$} -- (-4,0);
\filldraw  (-9.7, 0.1) circle (0.3pt);

\draw (-10.5,0) -- (-8,0);

\draw[fill=gray!15] (-7.59, 0.21) -- (-6,1) -- (-4, 0) --(-7.59, 0.21);
\draw (-3.25,-0.3) node[anchor=south] {\scriptsize $C$};
\draw (-6,0.29)node[anchor=north]{$\ell$};
\draw[fill=gray!60] (-10.5, 0) node[anchor=east] {\scriptsize $B$} -- (-9.7, 0.15) node[anchor=south]{\scriptsize $A$} -- (-4,0) -- (-10.5, 0);
\draw[densely dotted, thick] (-9.7,0.1) -- (-9.73,0);
\draw (-9.73,0)node[anchor=north]{\scriptsize enlarged negative area \# 1};
\draw[densely dotted, thick] (-3.7,-0.11) -- (-3.7, -0.15) -- (-3.8, -0.1);
\draw (-3.7, -0.45) node[anchor=west]{\scriptsize enlarged negative area \# 2};
\filldraw (-3.7,-0.11) circle (0.3pt);
\filldraw (-4,0) circle (0.3pt);
\draw[densely dotted, thick] (-8,0) -- (-7.59,0.21);
\draw (-6,-1) node[anchor=north]{\small FIGURE 4};
\draw (-6,0.55) node[anchor=north]{\scriptsize shrinked positive area};
\draw[densely dotted, thick] (-9.7, 0.1) node[anchor=east] {\scriptsize $p$} -- (-3.7,-0.11) node[anchor=south] {\scriptsize $q$};
\draw (-4,0) node[anchor=south] {\scriptsize $O$} -- (-8,0);

\end{tikzpicture}
\end{center}

Again due to the convex hulls above, the negative area formed in the stage $n$ by line $pq$ with $J_{n,1}$ and $J_{n,4}$ is at most $\frac{R_n (1+ 2\cos \theta_n)R_{n+1} \sin \theta_{n+1}}{2} + \frac{R_n R_{n+1}\sin \theta_{n+1}}{2}$. In fact, $pq$ line and $J_{n,1}$ can form a negative area at most the area of the triangle $ABO$, whose area is $\frac{R_n (1+ 2\cos \theta_n)R_{n+1} \sin \theta_{n+1}}{2}$. Similarly, $pq$ line and $J_{n,4}$ can form a negative area at most the area of the triangle $CDO$, whose area is $\frac{R_n R_{n+1}\sin \theta_{n+1}}{2}$. Thus 
$$
I\!I\!I \geq \frac{1+4\cos \theta_{n+1}}{2+4\cos \theta_{n+1}}S_n - \frac{R_n (1+ 2\cos \theta_n)R_{n+1} \sin \theta_{n+1}}{2} - \frac{R_n R_{n+1}\sin \theta_{n+1}}{2}.
$$

Therefore
\begin{align*}
|NH(p^{-1}q)^2|&\geq I + I\!I \geq I\!I\!I + I\!I\\
&\geq \left(\frac{1+4\cos \theta_{n+1}}{2+4\cos \theta_{n+1}}S_n - \frac{R_n (1+ 2\cos \theta_n)R_{n+1} \sin \theta_{n+1}}{2} - \frac{R_n R_{n+1}\sin \theta_{n+1}}{2}\right) -\sum_{j=1}^\infty 2\cdot 6^{j-1} S_{n+j} \\
&= R_n^2 \cdot \left( \frac{(1+4\cos \theta_{n+1}) \sin \theta_n \cos \theta_n}{2+4\cos \theta_{n+1}} - \frac{\sin \theta_{n+1}(1+\cos \theta_n)}{2+4\cos \theta_{n+1}} - \sum_{j=1}^\infty 2 \cdot 6^{j-1} \cdot \frac{\sin \theta_{n+j} \cos \theta_{n+j}}{\Pi_{k=n+1}^{n+j} (2+4\cos \theta_k)^2} \right) \\
&\geq R_n^2 \sin \theta_n \left( \frac{(1+4\cos \theta_{n+1})  \cos \theta_n}{2+4\cos \theta_{n+1}} - \frac{1+\cos \theta_n}{2+4\cos \theta_{n+1}} - \sum_{j=1}^\infty 2 \cdot 6^{-j-1} \cdot \frac{\cos \theta_{n+j}}{\prod_{k=n+1}^{n+j} (\frac{1}{3}+\frac{2}{3}\cos \theta_k)^2} \right).
\end{align*}
By (\ref{3}), $\sum \theta_n^2 <1$, and thus 
$$
\prod_{k=n+1}^{n+j} (\frac{1}{3}+\frac{2}{3}\cos \theta_k)^2 \geq \prod_{k=n+1}^{n+j} (1-\frac{1}{3} \theta_k^2)^2 \geq \prod_{k=n+1}^{n+j} (1-\frac{2}{3} \theta_k^2) \geq 1 - \frac{2}{3} \sum_{k=n+1}^\infty \theta_k^2 >\frac{1}{3}.
$$
Since each $\theta_n <\frac{1}{2}$, $\cos \theta_n > \frac{\sqrt{3}}{2}$ for all $n$. Therefore
\begin{align*}
|NH(p^{-1}q)^2|&\geq R_n^2 \sin \theta_n \left( \frac{4\cos \theta_{n+1} \cos \theta_n - 1}{2+4\cos \theta_{n+1}} - \sum_{j=1}^\infty  6^{-j} \right) \\
& \geq R_n^2 \sin \theta_n \left( \frac{2}{6} - \frac{1}{5} \right) \sim R_n^2 \theta_n. 
\end{align*}
\end{proof}

\begin{proof}[Proof of Theorem \ref{5.1}]
Let $\alpha \in (0,1)$. We will show $L^2$-unboundedness of $T^{2\alpha}$. Let $$
\theta_n = \frac{c}{n^{1/\alpha}},
$$
where $c:=c(\alpha)$ is chosen small enough so that \eqref{3} holds. Let $E:=E(\theta_n))$ as in the construction in the beginning of this section.

Denote $F_n= \{w\in E: w_n=1\}$. For any $p\in E$, denote 
$$
E_{n,p} = \{w\in E: w_i = p_i \text{ for } 1\leq i\leq n-1, w_n=4\}.
$$
Note $E_{n,p}$ has measure at least the length of some line segment $J_{n,i}$, which is $R_n$. By a similar argument as in \cite[Lemma 2.7]{seanmarkov}, $p$ is contained in the convex hull of the $6$-gonal replacement of the $n$-stage line segment that corresponds to $p$, and $q$ is contained in the convex hull of the $6$-gonal replacement of the $n$-stage line segment that corresponds to $q$. Thus
$$
\|p^{-1} q\|\lesssim R_{n-1} = R_n (2+ 4\cos \theta_n) \lesssim R_n.
$$
By Lemma \ref{5.4}, for $p\in F_n$,
\begin{align*}
\int_{E_{n,p}}K_{2 \alpha}(p^{-1}q)\,dq =\int_{E_{n,p}}\frac{NH(p^{-1}q)^{2\alpha}}{\|p^{-1}q\|^{2\alpha+1}}\,dq \gtrsim \int_{E_{n,x}}\frac{R_n^{2\alpha} \theta_n^\alpha}{R_n^{2\alpha+1}}\,dy \gtrsim \theta_n^\alpha.
\end{align*}
Note $\{F_n \times E_{n,p}\}_n$ are disjoint subsets in $E\times E$. Therefore
\begin{align*}
\langle T^{2 \alpha} \chi_E, \chi_E \rangle = \iint K_{2\alpha}(p^{-1}q)\,dq \geq \sum_n \int_{F_n} \int_{E_{n,p}} K_{2\alpha}(p^{-1}q)\,dq\,dp  \gtrsim \sum_n \theta_n^\alpha.
\end{align*}
Hence, 
$\langle T^{2\alpha} \chi_E, \chi_E \rangle =\infty.$ 
\end{proof}

\section{A SIO which is $L^2$-bounded on a $1$-unrectifiable set}
\label{sec:mainthm2}
In this section, we will prove Theorem \ref{mainthm2}. Let $S \subset [0,1]$ be any Ahlfors $\frac{1}{2}$-regular set. Let $f : [0,1] \to \He$ be defined as
$$
f(t)=(t,0,t), \quad t \in [0,1].
$$
Note that for $s,t \in [0,1]$, we have
\begin{align*}
  \|f(t)^{-1}f(s)\| = \|(s-t,0,s-t)\| \asymp |s-t|^{1/2}.
\end{align*}
Thus, it follows that $E:=f(S) \subset \He$ is an Ahlfors 1-regular set. Note that $\diam(E) \leq 2$. We will first show that the set $E$ is purely $1$-unrectifiable.

\begin{lem}
 The set $E$ is purely 1-unrectifiable.
\end{lem}

\begin{proof} It follows by \cite[Theorem 3.14]{MSS} and \cite{speight},see also \cite[Remark 4.109]{MR3587666}, that a set $E \subset \He$ is $1$-rectifiable if and only if for $\mathcal{H}^1$-a.e. $p \in E$, the set of the tangent measures of $\mathcal{H}^1|_{E}$ at $p$ consists only of measures of the form $c\mathcal{H}^1|_{L}$ where $L$ is a $1$-dimensional subgroup of $\He$, i.e. a set of the form $\{(at,b t,0):t \in \R\}$ for some $a,b \in \R$ not both zero.

  For $m \in \R$, let $L_m = \{(t,0,mt) : t \in \R\}$. As $E \subset L_1$, we get that $p^{-1}E \subset L_1$ for all $p \in E$. Then $\delta_\lambda(L_1) = L_\lambda$ for $\lambda > 0$. In particular, for $\lambda > 0$ small,
  \begin{align*}
    \delta_{\lambda^{-1}}(p^{-1}E \cap B(0,\lambda)) \subset \delta_{\lambda^{-1}}(L_1 \cap B(0,\lambda)) = L_{\lambda^{-1}} \cap B(0,1), \qquad \forall p \in E.
  \end{align*}
  This tells us that any tangent of a point of $E$ must converge to a set on the $z$-axis. Thus, $E$ has to be purely $1$-unrectifiable. 
\end{proof}

With the set $E$ now identified, we restate Theorem 1.3 accordingly.

\begin{thm}\label{a}
The SIO associated to $E:=f(S)$ and $K_b$ as in \eqref{eq:kerunrec} is bounded in $L^2(E)$; i.e. the truncated singular integrals $$T^b_\ve f\,(p):=\int_{E \stm B(p,\ve)} K_b(p^{-1} \cdot q) f(q) d \mathcal{H}^1(q)$$ are uniformly bounded in $L^2(E)$. Moreover, for every $p \in E$ it holds that $p^{-1}E \subset \supp (K_b)$.
\end{thm}
To prove Theorem \ref{mainthm2}, we first need the following lemma:
\begin{lem}\label{1}
  $\sup_{p \in E} \int_E K_b(p^{-1}q) ~d\cH^1(q) < \infty$.  
\end{lem}

\begin{proof}
  We will left-translate $E$ by $p^{-1}$ so that we can assume $p = 0$. Our estimates will not use the precise choice of $p$.

  For $0 < s < t$, let $A(s,t) = B(0,t) \setminus B(0,s)$. 
  For $k \geq 0$ and $q \in A(2^{-k},2^{-k+1})$, we have that $\|q\| \geq 2^{-k}$. On the other hand, letting $L = \{(t,0,t) : t \in \R\}$, we have that
  \begin{align*}
    E \cap A(2^{-k},2^{-k+1}) \subset L \cap B(0,2^{-k+1}) \subseteq \left\{(t,0,t) : |t| \leq 2^{-2k+2} \right\}.
  \end{align*}
  This gives that
  \begin{align}
    K_b(q) \leq \frac{2^{-2k+2}}{2^{-2k}} = 4, \qquad \forall q \in E \cap A(2^{-k},2^{-k+1}). \label{eq:K-x-bound}
  \end{align}

  As $\diam(E) \leq 2$, we have
  \begin{align*}
    \int_E K_b(q) ~d\cH^1(q) = \sum_{k=0}^\infty \int_{E \cap A(2^{-k},2^{-k+1})} K_b(q) ~d\cH^1(q) \leq 4 \sum_{k=0}^\infty \cH^1(E \cap B(0,2^{-k+1})) \lesssim \sum_{k=0}^\infty 2^{-k+1} \leq 4.
  \end{align*}
  The implicit constant in the penultimate inequality depends only on the Ahlfors regularity constant of $E$ and not on our initial choice of $p$. This proves the lemma.
\end{proof}

\begin{proof}[Proof of Theorem \ref{a}]
Let $\Delta=\{Q\}$ be the Christ cubes of $E$; see \cite{christ} and also \cite[Section 4]{CLZh}. By Lemma \ref{1}, and since $K_b$ is non-negative,
\begin{align*}
\|T^b_\ve\chi_Q\|_{L^2(Q)} &=\Big(\int_Q \Big| \int_{Q\backslash B(p,\ve)} K_b(p^{-1}q)\,d\mathcal{H}^1(q)\Big|^2\,d\mathcal{H}^1(p)\Big)^{1/2}\\
&\leq \Big(\int_Q \Big| \int_E K_b(p^{-1}q)\,d\mathcal{H}^1(q)\Big|^2\,d\mathcal{H}^1(p)\Big)^{1/2} \lesssim |Q|^{1/2},
\end{align*}
uniformly in $\ve>0$ and $Q\in \Delta$. Therefore, the T1 theorem implies that $T_\ve^b$ are bounded on $L^2(E)$ uniformly in $\ve$.
\end{proof}

\bibliographystyle{plain}
\bibliography{ref}
\end{document}